\documentclass[reqno,12pt]{article}

\usepackage[english, russian]{babel}
\usepackage{a4wide}
\usepackage{amsmath,stmaryrd} 
\usepackage{amssymb}
\usepackage{amsthm}
\usepackage[utf8]{inputenc} 
\usepackage{graphicx} 
\usepackage{xcolor}
\usepackage[normalem]{ulem}

\numberwithin{equation}{section}

\newtheorem{theo}{Theorem}[section]
\newtheorem{conj}[theo]{Conjecture}
\newtheorem{coro}[theo]{Corollary}
\newtheorem{prop}[theo]{Proposition}

\newtheorem{lem}[theo]{Lemma}
\newtheorem{defi}[theo]{Definition}

\theoremstyle{remark}
\newtheorem{Remark}{Remark}

\newtheorem*{Remark*}{Remark}
\newtheorem*{Remarks*}{Remarks}

\makeatletter
\newcommand*{\house}[1]{%
 \mathord{%
 \mathpalette\@house{#1}%
 }%
}
\newcommand*{\@house}[2]{%
 \dimen@=\fontdimen8 %
 \ifx#1\scriptscriptstyle\scriptscriptfont
 \else\ifx#1\scriptstyle\scriptfont
 \else\textfont\fi\fi
 3 %
 \sbox0{%
 $#1%
 \vrule width\dimen@\relax
 \overline{%
 \kern2\dimen@
 \begingroup 
 #2%
 \endgroup
 \kern2\dimen@
 }%
 \vrule width\dimen@\relax
 \mathsurround=1.5\dimen@ 
 $%
 }%
 \ht0=\dimexpr\ht0-\dimen@\relax
 \dp0=\dimexpr\dp0+2\dimen@\relax
 \vbox{%
 \kern\dimen@ 
 \copy0 %
 }%
}

\newcommand{\N}{\mathbb{N}}
\newcommand{\Z}{\mathbb{Z}}
\newcommand{\Q}{\mathbb{Q}}
\newcommand{\R}{\mathbb{R}}
\renewcommand{\C}{\mathbb{C}}

\newcommand{\E}{\EE}
\newcommand{\Qbar}{\overline{\mathbb Q}}

\newcommand{\eps}{\varepsilon}

\newcommand{\calE}{\mathcal E}
\newcommand{\calP}{\mathcal P}
\newcommand{\calR}{\mathcal R}

\newcommand{\NGA}{{\rm NGA}}

\newcommand{\EE}{{\bf E}}
\newcommand{\DD}{{\bf D}}
\newcommand{\VV}{{\bf V}}
\newcommand{\GG}{{\bf G}}

\newcommand{\al}{\alpha}
\newcommand{\fgo}{\mathfrak{f}} 
\newcommand{\tantie}{\foreignlanguage{russian}{\CYREREV}}

\newcommand{\antie}{\text{\tantie}}

\newcommand{\Stheta}{{\mathcal S}{_\theta}}

\begin{document}

 \selectlanguage{english}

\title{The ring of values of $E$-functions is not a field}
\date\today
\author{S. Fischler and T. Rivoal}
\maketitle

\begin{abstract} Let $\E$ be the ring of values taken at algebraic points by $E$-functions. The exponential function and the Bessel function $J_0$ are classical examples of $E$-functions. We prove that $\E$ is not a field. More precisely, we prove that it is neither noetherian nor factorial, and construct non-units of $\E$ explicity, including $e+1$, $J_0(1)$, and ${}_1F_1(a;b;1)$ for rational parameters $a$ and $b$ such that $0<a<1$ and $0<b-a<1$.
\end{abstract}

\section{Introduction}

Siegel \cite{siegel} introduced in 1929 the class of $E$-functions. A power series $f(x)=\sum_{n=0}^\infty \frac{a_n}{n!}x^n$ is an $E$-function if it is solution of a homogeneous linear differential equation with coefficients in $\Qbar(x)$ and the coefficients $a_n$ are algebraic numbers, such that there exists a constant $C>0$ with the following properties, for any $n\geq 0$:
\begin{itemize}
 \item[$(i)$] All Galois conjugates of $a_n$ have modulus less that $C^{n+1}$.
 \item[$(ii)$] $a_0$, \ldots, $a_n$ have a positive common denominator less that $C^{n+1}$.
\end{itemize}
In this setting, $\sum_{n=0}^\infty a_n x^n$ is called a $G$-function, and $\sum_{n=0}^\infty n! a_n x^n$ is called an \antie-function (or an arithmetic Gevrey series of order 1). 

\bigskip

$E$-functions are entire; here and throughout this paper, we see $\Qbar$ as embedded in $\C$.
Examples of $E$-functions include the (normalized) Bessel function $x^{-\alpha}\Gamma(\alpha+1)J_\alpha(x)=\sum_{n=0}^\infty (-1)^n (x/2)^{2n}/(n!(\alpha+1)_n)$ for $\alpha\in\Q$ and the $E$-hypergeometric series ${}_pF_q[x^{q-p+1}]$ with rational parameters, when $q\geq p \geq 1$. We recall that hypergeometric series are defined as follows:
$$
{}_pF_q[a_1, \ldots, a_p;b_1,\ldots, b_q;x]:=\sum_{n=0}^\infty \frac{(a_1)_n\cdots (a_p)_n}{n!(b_1)_n\cdots (b_q)_n}x^n
$$
where $p,q$ are non-negative integers, $a_1, \ldots, a_p\in \mathbb C$, $b_1, \ldots, b_q\in \mathbb C\setminus \mathbb{Z}_{\le 0}$, and $(\alpha)_n:=\alpha(\alpha+1)\cdots (\alpha+n-1)$ is Pochhammer's symbol.

$E$-functions appear in expressions of the form $\int_\sigma e^{-zf}\omega$ where $f$ is a regular function on an algebraic variety $X$ defined over $\Qbar$, $\sigma$ is a rapid decay cycle on $X(\C)$ and $\omega$ is an algebraic differential form on $X$; moreover any $E$-function is conjectured to appear in this way 
(see \cite[\S 1.6]{Fresannotes}).

\bigskip

The set of values $f(x_0)$, where $f$ is an $E$-function and $x_0$ is an algebraic number, is a ring denoted by $\EE$ and studied in \cite{ateo} (see also \cite{XUPS}); in this definition we may assume $x_0=1$ since $f(x_0x)$ is again an $E$-function. Values of $E$-functions are of special interest because of their conjectural relation with exponential periods \cite{FJ}, in the same way as values of $G$-functions with periods (see \S \ref{secopen} below).

\bigskip

A difficult problem is to prove that a given number {\em does not} belong to $\EE$. A way to do this is to prove that it is extremely well approximated by rationals \cite{eliouville,muegaledeux}. But when it comes to numbers appearing naturally in analysis, conjecturally this method cannot be applied. In this paper we provide a new way to prove that a given element does not belong to $\EE$. This enables us to prove that $\EE$ is not a field, a folklore conjecture. More precisely, our main result is the following.

\begin{theo} \label{th1}
The ring $\EE$ is not a field. More precisely, it is neither noetherian nor factorial, and the following elements of $\EE$ are not units:
\begin{itemize}
\item[$\bullet$] $e+1$,
\item[$\bullet$] $J_0(1)$,
\item[$\bullet$] $ {}_1F_1[a;b;1]$ with rationals $a$, $b$ such that $0<a<1$ and $0<b-a<1$.
\end{itemize}
\end{theo}
The argument 1 at which the above examples are evaluated plays no special role. The important point is that we introduce a new tool that enables us to determine explicit elements of $\E\setminus \E^\times$.

\bigskip

As noticed in \cite[\S 5]{ezeros}, knowing that $\EE$ is not a field has a consequence on the structure of the ring $\calE$ of $E$-functions: the ideal $I$ of $\calE$ generated by $f(x)=e^x+1$ and $g(x)=x-1$ is not principal (and in this example, $f$ can be replaced with any $E$-function such that $f(1)\not\in\EE^\times$). Indeed, if $I$ is generated by an $E$-function $h$, then $h$ divides $x-1$ so that $h$ is a unit or $h(x)=(x-1)u(x)$ for some unit $u$ of $\calE$. The latter is impossible since $h$ divides $f$; therefore $1\in I$ and 
there exist $f_1$, $g_1$ such that $
ff_1+gg_1=1$. Taking $x=1$ we obtain $1/(e+1)=f_1(1)\in\EE$, which is a contradiction. In particular, even though $e^x+1$ and $ x-1$ are coprime (see \cite{ezeros}), there is no Bezout relation between them. 

\bigskip

It would be very interesting to study this ring $\EE$ further, for instance factorization properties. A few hints in this direction are given at the end of this text. But the first step would be to determine exactly its units. In this respect, the following conjecture is natural (see \cite[Problème 6]{XUPS} and \cite[Theorem 5]{RiRopreprint}):

\begin{conj} \label{conj:1} Units of $\EE$ are exactly the $ \alpha e^{\beta }$ with $\alpha,\beta \in\Qbar$ and $\alpha\neq0$.
\end{conj}

We prove the following result in this direction.
\begin{theo} \label{th2} 
Let $\xi = \sum_{j=1}^p \alpha_j e^{\beta_j}$ with $\alpha_1,\ldots,\alpha_p, \beta_1,\ldots,\beta_p\in\Qbar$. Then $\xi$ is a unit of $\EE$ if, and only if, $\xi$ is of the form $ \alpha e^{\beta }$ with $\alpha,\beta \in\Qbar$ and $\alpha\neq0$.
\end{theo}

In particular, $e+1$ is not a unit of $E$ (as mentioned above).

In other words, Theorem~\ref{th2} means that $\frac1{\sum_{j=1}^p \alpha_j e^{\beta_j}}$ does not belong to $\EE$, except in the obvious case where the denominator is of the form $ \alpha e^{\beta }$. We will prove more generally that if a number of the form 
$$
\frac{\sum_{j=1}^p \alpha_j e^{\beta_j}}{\sum_{j=1}^{p'} \alpha'_j e^{\beta'_j}}, \quad \alpha_j, \alpha_j', \beta_j, \beta_j' \in \Qbar,
$$ 
belongs to $\EE$, then the denominator ``divides'' the numerator: this number can be written as $\sum_{j=1}^{p''} \alpha''_j e^{\beta''_j}$ with $\alpha_j'', \beta_j''\in \Qbar$ (see Corollary~\ref{coroth3} in \S \ref{secmorph}).

\bigskip

The main tool used in this paper is the asymptotic expansion of an $E$-function $f$ in a large sector bisected by a fixed non-anti-Stokes direction $\theta$, with $-\pi/2<\theta<\pi/2$; by {\em large} we mean that its opening is greater than $\pi$. The important point in our approach is that it provides a unique decomposition of $f$ as a finite sum 
\begin{equation}\label{eq:decompf}
f(x) = \sum_{\rho\in\Qbar} f_\rho(x) e^{\rho x}
\end{equation}
where, for each $\rho\in\Qbar$, $f_\rho$ is a function holomorphic on an open subset of $\C$ containing the point $x=1$ (we refer to Theorem~\ref{thdcp} for the properties of this decomposition). 

We prove that each {\em value} $\xi\in\EE$ admits a similar decomposition. Indeed, we attach to $\xi$ an exponential polynomial $\Phi_\theta(\xi)$ defined by
$$\Phi_\theta(\xi)(x)= \sum_{\rho\in\Qbar} f_\rho(1)e^{\rho x}
$$
if $\xi=f(1)$ for some $E$-function $f$ of which $\theta$ is not an anti-Stokes direction and $f$ has a decomposition as in \eqref{eq:decompf}.
Taking $x=1$ gives a decomposition of $\xi$:
$$ \xi = \sum_{\rho\in\Qbar} f_\rho(1)e^{\rho }.
$$
The crucial point here is that this decomposition (and, more generally, $\Phi_\theta(\xi)$) depends only on $\xi$ and $\theta$, not on the choice of $f$. This follows from Beukers' theorem \cite{beukers} (proved by André \cite{andre} in the case of rational coefficients): if $h$ is an $E$-function that vanishes at 1, then $h(x)/(x-1)$ is an $E$-function. 

We prove that $\Phi_\theta$ is an injective ring morphism (see Proposition~\ref{propfonda} for details), and this enables us to deduce the results announced above.

\bigskip

Siegel's definition of $E$-functions is in fact more general that the above one, with both upper bounds $C^{n+1}$ replaced by $n!^{\varepsilon}$ for all $\varepsilon>0$ provided $n\ge N(\varepsilon)$. For such $E$-functions, we can extend the definition of the set $\E$ and it is still a ring, the units of which are also conjectured to be of the form given in Conjecture~\ref{conj:1}. Then Theorem~\ref{th1} holds {\em verbatim} in this more general situation with the same proof, {\em mutatis mutandis} so do all the results stated and proved in \S \ref{secmorph}. Indeed, we shall use the existence of an $E$-operator with certain ``good'' properties, Beukers' theorem, and properties of asymptotic expansions of $E$-functions described in Theorem~\ref{thdcp}: these properties also hold for $E$-functions in Siegel's original sense by results of Andr\'e \cite{andreasens} and Lepetit \cite{lepetit}.

\bigskip

The structure of this paper is as follows. In \S \ref{secasy} we recall the properties of asymptotic expansions of $E$-functions, and sum up in Theorem~\ref{thdcp} the most useful ones in our approach. Then we construct $\Phi_\theta$ in \S \ref{secmorph}, and use it to prove Theorem~\ref{th2}. 
We make explicit computations in \S \ref{secexplicite}: using integral formulas for the Borel-Laplace summation in a special case, we work out examples related to Kummer's function and Bessel's function $J_0$. We gather all ingredients in \S \ref{sec5} and 
 complete the proof of Theorem~\ref{th1}. At last, \S \ref{secopen} is devoted to open questions.

\bigskip 

\noindent {\em Statement on the use of Artificial Intelligence:} no AI tools were used to prepare this paper, except to find reference \cite{Fresannotes}.

\section{Asymptotic expansions} \label{secasy}

Let us recall the main properties of asymptotic expansions of $E$-functions in large sectors. We will sum up in Theorem~\ref{thdcp}, at the end of this section, the properties we shall use in the rest of the paper.

\bigskip

To begin with, the following definition appears in \cite[\S 4.1]{ateo} (see also \cite[p. 30]{Ramis}).

\begin{defi} \label{defiasy}
Let $\theta\in\R$, and $\Sigma\subset\C$, $S\subset\Q$, $T\subset\N$ be finite subsets. Given
complex numbers $ c_{\rho, \al,i,n}$, we write
\begin{equation} \label{eqasy1}
f(x) \approx \sum_{\rho\in\Sigma} e^{\rho x} 
 \sum_{\al\in S } \sum_{i\in T } \sum_{n=0}^{\infty} c_{\rho, \al,i,n}x^{-n-\al}(\log(1/x))^i
 \end{equation}
and say that the right hand side is the (generalized) asymptotic expansion of $f(x)$ in a large sector bisected by 
the direction $\theta$, if there exist $\eps, R, B, C > 0$ and, for any $\rho\in\Sigma$, a function $f_\rho(x)$ holomorphic on
$$U = \Big\{x\in\C, \,\,|x|\geq R, \, \, \theta-\frac{\pi}2-\eps \leq \arg(x) \leq \theta+\frac{\pi}2+\eps \Big\},
$$
such that 
$$f(x) = \sum_{\rho\in\Sigma} f_\rho(x)e^{\rho x} 
 $$
 and
$$
\Big| f_\rho(x) - \sum_{\al\in S } \sum_{i\in T } \sum_{n=0}^{N-1} c_{\rho, \al,i,n}x^{-n-\al}(\log(1/x))^i\Big| \leq C^N N! |x|^{B-N} 
$$
for any $x\in U$ and any $N\geq 1$.
\end{defi}

In other words (see \cite[\S\S 2.1 and 2.2]{Ramis}), for any $\rho\in\Sigma$, the formal power series 
\begin{equation} \label{eqasy2}
\widehat f_\rho(x)= \sum_{\al\in S } \sum_{i\in T } \sum_{n=0}^{N-1} c_{\rho, \al,i,n}x^{-n-\al}(\log(1/x))^i
 \end{equation}
 is 1-summable in the direction $\theta$ and its sum is $f_\rho(x)$. Accordingly, by a result of
 Watson (see \cite[p. 33]{Ramis}), the sum $f_\rho(x) $ is determined by $\widehat f_\rho(x)$.
 Therefore the asymptotic expansion on the right hand side of \eqref{eqasy1} determines the function $f(x)$ 
(up to analytic continuation). The converse is also true, as the following lemma shows (see \cite[Lemma 1]{ateo}).
 
 \begin{lem} \label{lemasyunique}
 A given function $f $ can have at most one asymptotic expansion in the sense of Definition~\ref{defiasy}.
\end{lem}

From now on, we restrict to the case where $f$ is an $E$-function. Then André \cite{andre} gave a complete picture that we state now.

For any real $\theta$ (except finitely many directions mod $2\pi$, called {\em anti-Stokes}), $f$ has an asymptotic expansion \eqref{eqasy1} in a large sector bisected by $\theta$. Given such a $\theta$, the functions $ f_\rho$ of Definition~\ref{defiasy} and the formal power series $\widehat f_\rho$ of Eq. \eqref{eqasy2} exist and are unique. Moreover André proved \cite{andre} that $\widehat f_\rho$ belongs to $\NGA_1\{1/x\}$ (see \S \ref{secopen} for the definition of this set), but this property won't be used in this paper. 

Another important property is that $f$ and the functions $e^{\rho x} f_\rho(x)$ are solutions of an $E$-operator, of which $0$ is the only possible finite singularity. In particular, the functions $f_\rho$ can be analytically continued to a large sector bisected by $\theta$, in which the asymptotic expansion of $f$ holds.
 In what follows we shall assume that $-\pi/2<\theta<\pi/2$, so that the point $x=1$ belongs to this sector and $f_\rho(1)$ can be obtained by performing 1-summation of $\widehat f_\rho(x)$ and letting $x=1$.

\bigskip
 
Finally, let us sum up the above properties \cite{andre, Ramis}. 

\begin{theo} \label{thdcp}
Let $f$ be an $E$-function, and $\theta$ be a direction which is not anti-Stokes; assume that $-\pi/2<\theta<\pi/2$.

Then there exist $\alpha>0$ and a unique family $(f_\rho)_{\rho\in\Qbar}$ of functions such that:
\begin{itemize}
\item The functions $f_\rho$ are holomorphic on the open set $D_\alpha$ defined by $x\neq 0$ and $\theta-\frac\pi2-\alpha<\arg(x) <\theta+\frac\pi2+\alpha$, and (except for finitely many $\rho$) they are identically zero.
\item For any $x\in D_\alpha$, we have $f(x) = \sum_{\rho\in\Qbar} f_\rho(x)e^{\rho x} $.
\item Each $ f_\rho(x)$ has an asymptotic expansion $\widehat f_\rho (x) \in \NGA_1\{1/x\}$, valid as $|x|\to+\infty$ in $D$.
\item If $\theta$ and $\theta'$ are such directions, and if no anti-Stokes direction lies between $\theta$ and $\theta'$, then the functions $f_\rho$ associated with $\theta$ are the same as with $\theta'$.
\end{itemize}
\end{theo}

\section{A ring morphism} \label{secmorph}

In this section we construct (in Proposition~\ref{propfonda}) an injective ring morphism $\Phi_\theta$ which is the main new feature of this paper. 

From it we will deduce Corollary~\ref{coroth3} (of which Theorem~\ref{th2} stated in the introduction is a consequence), and then a necessary condition for a number to be a unit of $\EE$ (see Corollary~\ref{corpasunite}). This condition will be used in \S \S \ref{secexemple} and~\ref{subsecbessel} to give examples of non-units.

\bigskip

We denote by $\mathcal{R}$ the ring of exponential polynomials with constant complex coefficients and algebraic phases. In other words, $P\in \mathcal{R}$ means $P(x)=\sum_{\rho\in\Qbar}\lambda_\rho e^{\rho x}$ where the coefficients $\lambda_\rho$ are complex numbers and the phases $\rho$ are algebraic numbers; here and throughout this paper, all sums indexed by $\rho\in\Qbar$ are finite sums, since all terms are zero except finitely many of them.

\begin{prop}\label{propfonda}
Let $\theta\in\R$ with $-\pi/2<\theta<\pi/2$. Then there exists an injective ring morphism $\Phi_\theta: \EE\to\mathcal{R}$ such that: 
\begin{itemize}
\item[$(i)$] For any $\xi\in \EE$, $(\Phi_\theta(\xi))(1)=\xi$.
\item[$(ii)$] Let $f$ be an $E$-function. Write $f(x)=\sum_{\rho\in\Qbar} f_\rho(x)e^{\rho x}$ in a large sector bisected by $\theta+\eps$, as in Theorem~\ref{thdcp}, with $\eps>0$ sufficiently small (so that this expansion is independent from $\eps$; it is also valid for $\eps=0$ if $\theta$ is not anti-Stokes). Then, 
$$
\Phi_\theta(f(1))= \sum_{\rho\in\Qbar} f_\rho(1)e^{\rho x}\in \mathcal{R}.
$$
\end{itemize}
\end{prop}

\begin{Remark} \label{rempolexp}
 If $\xi =\sum_{j=1}^p \alpha_j e^{\beta_j}$ with $\alpha_1,\ldots,\alpha_p,\beta_1,\ldots,\beta_p\in\Qbar$, then $\Phi_\theta(\xi) =\sum_{j=1}^p \alpha_j e^{\beta_jx} $ (for any $\theta$ such that $-\pi/2<\theta<\pi/2$). Indeed $f(x)= \sum_{j=1}^p \alpha_j e^{\beta_j x}$ is an $E$-function such that $f(1)=\xi$, and (assuming, as we may, that the $\beta_j$ are pairwise distinct) its decomposition is given by $f_\rho(x)= \alpha_j$ if $\rho=\beta_j$ for some $j$, and $f_\rho(x)= 0$ otherwise.
\end{Remark}

\begin{proof} Since we need to allow the case where the direction $\theta$ is anti-Stokes for $f$, we shift it a little by considering $\theta+\eps$ for a sufficiently small $\eps>0$; in the case where $\theta$ is not anti-Stokes for $f$, this does not change anything. 
To begin with, let us prove that the decomposition $f(x)=\sum_{\rho\in\Qbar} f_\rho(x)e^{\rho x}$ given by Theorem~\ref{thdcp}
 in a large sector bisected by $\theta+\eps$, with $\eps>0$ sufficiently small, is independent from $\eps$. Indeed, let $\theta_-$ be the maximal anti-Stokes direction such that $-\pi/2< \theta_- \leq \theta$; if there is no such anti-Stokes direction, we let $\theta_- = -\pi/2$. In a similar way, let $\theta_+$ be the least anti-Stokes direction such that $ \theta<\theta_+ <\pi/2$; if there is no such anti-Stokes direction, let $\theta_+ = \pi/2$. Then the expansion of $f$ given by Theorem~\ref{thdcp} in a large sector bisected by $\theta'$ is the same for all $\theta'$ such that $\theta_- < \theta' < \theta_+$. This includes the case of $\theta' = \theta + \eps$ for any sufficiently small $\eps>0$; if $\theta$ itself is not anti-Stokes, it includes also the case where $\theta' = \theta$.

\bigskip

Let us define $\Phi_\theta(\xi)$ for $\xi\in \EE$. We choose an $E$-function $f$ such that $f(1)=\xi$, and let $\Phi_\theta(\xi)= \sum_{\rho\in\Qbar} f_\rho(1)e^{\rho x}$ where $f(x)=\sum_{\rho\in\Qbar} f_\rho(x)e^{\rho x}$ in a large sector bisected by $\theta+\eps$, with $\eps>0$ sufficiently small. In this way, property $(ii)$ holds. The point is to prove that $\Phi_\theta(\xi)$ depends only on $\xi$, and not on the $E$-function $f$ we have chosen.

\bigskip

Let $g$ be another $E$-function such that $g(1)=\xi$. Then $g-f$ vanishes at $1$, so that Beukers' theorem \cite{beukers} provides an $E$-function $h$ such that $f(x)=g(x)+(x-1)h(x)$. For a sufficiently small $\eps>0$, Theorem~\ref{thdcp} applies (separately) to the three $E$-functions $f$, $g$, $h$. Writing the expansions it provides in a large sector bisected by $\theta+\eps$ and, using uniqueness, we obtain $f_\rho(x)=g_\rho(x)+(x-1)h_\rho(x)$ for any $\rho\in\Qbar$. Since $h_\rho$ is holomorphic at 1, this implies $f_\rho(1)=g_\rho(1)$: the value $\Phi_\theta(\xi)$ computed using $g$ is the same as the one computed using $f$.

\bigskip

When $\xi\in\Qbar$ we may choose for $f$ the constant function equal to $\xi$, so that $\Phi_\theta(\xi)=\xi$. In particular, $\Phi_\theta(1)=1$.

\bigskip

Let $\xi_1,\xi_2\in\EE$, and $f_1,f_2$ be $E$-functions such that $f_1(1)=\xi_1$ and $f_2(1)=\xi_2$. For any sufficiently small $\eps>0$, the asymptotic expansion of $f_1+f_2$ in a large sector bisected by $\theta+\eps$ is the sum of those of $f_1 $ and $f_2$; this yields $
\Phi_\theta(\xi_1+\xi_2) = \Phi_\theta(\xi_1)+\Phi_\theta(\xi_2)$. 

On the other hand, for any sufficiently small $\eps>0$, the asymptotic expansion of $g=f_1f_2$ in a large sector bisected by $\theta+\eps$ is the product of those of $f_1$ and $f_2$ (see \cite[\S 2.1]{Ramis}). Accordingly, for any $\rho$, 
$$
g_\rho(x) = \sum_{\alpha+\beta=\rho} f_{1,\alpha}(x)f_{2,\beta}(x)
$$
where the sum is over all pairs $(\alpha,\beta)\in\Qbar^2$ such that $\alpha+\beta=\rho$. Taking $x=1$ yields
$$
\Phi_\theta(\xi_1\xi_2) = \sum_{\rho\in\Qbar} g_\rho(1)e^{\rho x}= \sum_{(\alpha,\beta)\in\Qbar^2} f_{1,\alpha}(1)f_{2,\beta}(1)e^{(\alpha+\beta) x}
= \Phi_\theta(\xi_1)\Phi_\theta(\xi_2).
$$

\bigskip

At last, let $\xi\in\EE$ and $f$ be an $E$-function such that $f(1)=\xi$. Choose $\eps>0$ sufficiently small, with $-\pi/2<\theta+\eps<\pi/2$. Then we have
$$
(\Phi_\theta(\xi))(1) = \sum_{\rho\in\Qbar} f_\rho(1)e^{\rho } = f(1).$$
This proves $(i)$; the fact that $\Phi_\theta$ is injective follows from it.
\end{proof}

Let us deduce from Proposition~\ref{propfonda} the following corollary, announced in the introduction and of which Theorem~\ref{th2} is a consequence.

\begin{coro}\label{coroth3} Let $p,p'\geq 1$ and $\alpha_1,\ldots,\alpha_p$, $\beta_1,\ldots,\beta_p$, $\alpha'_1,\ldots,\alpha'_{p'}$, $\beta'_1,\ldots,\beta'_{p'}$ be algebraic numbers. Assume that 
$$\xi_0 = \frac{\sum_{j=1}^p \alpha_j e^{\beta_j}}{\sum_{j=1}^{p'} \alpha'_j e^{\beta'_j}}
$$
belongs to $\EE$. Then $\xi_0 =\sum_{j=1}^{p''} \alpha''_j e^{\beta''_j} $ for some 
$p''\geq 1$ and $\alpha''_1,\ldots,\alpha''_{p''}$, $\beta''_1,\ldots,\beta''_{p''}\in\Qbar$.
\end{coro} 

\begin{proof}
To begin with, we may assume that $\beta_1,\ldots,\beta_p$ (resp. $\beta'_1,\ldots,\beta'_{p'}$) are pairwise distinct.

Write $\xi =\sum_{j=1}^p \alpha_j e^{\beta_j}$ and $\xi' =\sum_{j=1}^{p'} \alpha'_j e^{\beta'_j}$, so that $\xi_0 \xi' = \xi$. We fix $\theta=0$ (for instance) and apply the ring morphism $\Phi_0$ of Proposition~\ref{propfonda}: we obtain $\Phi_0(\xi_0)\Phi_0(\xi')=\Phi_0(\xi)$.
Writing $\Phi_0(\xi_0)=\sum_\rho \lambda_\rho e^{\rho x}$ where the $\lambda_\rho$ are complex numbers, we obtain
using Remark~\ref{rempolexp} for $\xi$ and $\xi'$:
$$
\Big( \sum_\rho \lambda_\rho e^{\rho x}\Big)
\Big( \sum_{j=1}^{p'} \alpha'_j e^{\beta'_jx}\Big) = 
\Big( \sum_{j=1}^p \alpha_j e^{\beta_jx}\Big).
$$

For any $\sigma\in\Qbar$, the coefficients of $e^{\sigma x}$ are equal so that:
$$
\sum_{j=1}^{p'}\lambda_{\sigma-\beta'_j} \alpha'_j = 
\left\{\begin{array}{l} \alpha_j\mbox{ if $\sigma=\beta_j$ for some $j$}\\0\mbox{ otherwise}\end{array}\right.
$$
This is a linear system of which the $\alpha'_j$ are coefficients and the 
$\lambda_\rho$ are unknows. It has a unique solution with $\lambda_\rho\in\C$. Now all $ \alpha'_j$ and all $ \alpha'_j$ are algebraic, so this solution consists in algebraic numbers too. Since $\xi_0 = \Phi_0(\xi_0)(1)=\sum_\rho \lambda_\rho e^{\rho }$ this concludes the proof of Corollary~\ref{coroth3}.
\end{proof}

\bigskip

The following corollary is a powerful tool to prove that a given number is not a unit of $\EE$; it is a step towards Conjecture~\ref{conj:1}. We will apply it in \S \S \ref{secexemple} and~\ref{subsecbessel} to give examples of non-units.

\begin{coro} \label{corpasunite}
Let $f$ be an $E$-function such that $f(1)$ is a unit of $\EE$. Let $\theta\in\R$ with $-\pi/2<\theta<\pi/2$; assume that $\theta$ is not an anti-Stokes direction for $f$, and denote by $\sum f_\rho(x)e^{\rho x}$ the expansion given by Theorem~\ref{thdcp} in a large sector bisected by $\theta$.

Then there exists a unique $\rho_0\in\Qbar$ such that $ f_{\rho_0}(1)\neq 0$.
\end{coro}

 \begin{proof} Since $f(1)$ is a unit of $\EE$, there exists an $E$-function $g$ such that $f(1)g(1)=1$. Since $\Phi_\theta$ is a ring morphism we have $\Phi_\theta(f(1))\Phi_\theta(g(1))=1$. Now the units of $\mathcal{R}$ are exactly~(\footnote{Indeed, any element of $\mathcal{R}$ is an entire function with exponential growth, so that it can be factorized in the form given by Hadamard's factorization theorem; since a unit of $\mathcal{R}$ has no zeros, it must therefore be of the stated form.})
 the functions of the form $\lambda e^{\rho_0 z}$ with $\lambda\in\C^\times$ and $\rho_0\in\Qbar$. Therefore $\Phi_\theta(f(1))$ is of this form: this concludes the proof.
\end{proof}

\section{Explicit computations} \label{secexplicite}

This section is devoted to explicit computations in special cases. We first recall in \S \ref{subsec41} the definition of the Borel-Laplace summation. Then we compute in \S \ref{subsec42} the Borel-Laplace summation of a ${}_2F_0$. This enables us to study the value $ {}_1F_1[a;b;1]$ in \S \ref{secexemple}. At last, we turn to the Bessel function $J_0$ in \S \ref{subsecbessel}.

\subsection{Definition of the Borel-Laplace summation}\label{subsec41}

There is a subtle point in our setting: we compute the Borel-Laplace summation (i.e., 1-summation) of functions of $1/x$, not of $x$. To make everything clear, let us 
recall the definition (see for instance \cite{Ramis}). We begin with a direction $\theta$ and an \antie-function in $1/x$ with no constant term, $\fgo (1/x)=\sum_{n=0}^\infty c_nx^{-n-1}$; in general this series is divergent for any $x$. The Borel transform of $\fgo (1/x)$ is
$$
 \mathcal{B}\!\left(\sum_{n=0}^\infty c_nx^{-n-1}\right)(\xi)
 :=\sum_{n=0}^\infty\frac{c_n}{n!}\xi^{n}.
$$
Since $\fgo (1/x)$ is an \antie-function, this is a $G$-function of the variable $\xi$ denoted by $g(\xi)$. In particular, it has moderate growth at infinity. If $\theta$ is not an anti-Stokes direction, the Borel-Laplace summation (i.e., 1-summation) in the direction $\theta$ is then given by 
\begin{equation}
 \label{eqBLs}
 \Stheta \fgo (1/x)
 :=\int_0^{e^{-i\theta}\infty}
 e^{-\xi x}g(\xi)\,d\xi
\end{equation}
for any $x$ in a large sector bisected by $\theta$ (i.e., such that $x\neq 0 $ and $\theta-\pi/2-\eps < \arg(x) < \theta+\pi/2+\eps $ for some $\eps>0$). We would like to draw the reader's attention on the fact that the integration path goes from $0$ to $e^{-i\theta}\infty$, because we are computing the Borel-Laplace summation of $\fgo (1/x)$ in a large sector bisected by $\theta$ (and not that of $\fgo (x)$). In the setting of this paper, we assume $|\theta|<\pi/2$ and evaluate Eq. \eqref{eqBLs} at $x=1$.

\subsection{Computation of the Borel-Laplace summation of a ${}_2F_0$}\label{subsec42}

In this section we consider
$$
 \fgo (1/x):=\frac1x \cdot {}_2F_0[\alpha,\beta;;1/x]
 =\sum_{n=0}^\infty\frac{(\alpha)_n(\beta)_n}{n!}x^{-n-1}
$$
with $\alpha, \beta \in \C$. If $\alpha$ and $\beta$ are rational then $\fgo $ is an \antie-function and we are in the setting of \S \ref{subsec41}. However this assumption is not necessary here: for any 
$\alpha, \beta \in \C$ the Borel-Laplace summation makes sense, and Eq. \eqref{eqBLs} holds. To make it explicit, we notice that 
 $$\mathcal{B} \!\left( \fgo (1/x)\right) (\xi)={}_2F_1[\alpha,\beta;1;\xi]=\sum_{n=0}^\infty\frac{(\alpha)_n(\beta)_n}{n!^2}\xi^n$$
so that Eq. \eqref{eqBLs} yields, for any $\theta$ such that $0<\vert \theta\vert<\pi/2$:
\begin{equation}\label{eq:BLplusbis}
 \Stheta \fgo (1)
 =\int_0^{e^{-i\theta}\infty}
 e^{-\xi}{}_2F_1[\alpha,\beta;1;\xi]\,d\xi.
\end{equation}
Here we use analytic continuation of $\xi \mapsto {}_2F_1[\alpha,\beta;1;\xi]$ outside the unit disk, in an open subset of $\C$ that contains the half-line $[0, e^{-i\theta}\infty)$. Since $\xi \mapsto {}_2F_1(\alpha,\beta;1;\xi)$ satisfies a fuchsian differential equation, 
it has moderated growth at infinity and the integral \eqref{eq:BLplusbis} is convergent.

We shall also be interested in $\tilde \fgo (1/x)=\fgo (-1/x)$. Its Borel transform is $-{}_2F_1[\alpha,\beta;1;-\xi]$ so that  
Eq. \eqref{eqBLs} yields, for any $\theta$ such that $\vert \theta\vert<\pi/2$:
\begin{equation}\label{eq:BLpluster}
 \Stheta\tilde \fgo (1)= -\int_0^{e^{-i\theta}\infty}
 e^{-\xi}{}_2F_1[\alpha,\beta;1;-\xi]\,d\xi.
\end{equation}

In the following lemma, we provide simpler expressions for $\Stheta \fgo (1)$ and $\Stheta\tilde \fgo (1)$, that will in particular enable us to prove  that they are non-zero.

\begin{lem}\label{lem:1summation}
Let $\alpha,\beta\in\C$ be such that $0<\Re(\alpha), \Re(\beta)<1$. We consider 
$$\fgo (1/x):=\frac1x\cdot {}_2F_0[\alpha,\beta;;1/x]
\quad \mbox{and } \quad 
\tilde \fgo (1/x) := \fgo (-1/x).
$$

\medskip
$(i)$
For any $\theta$ such that $0<\vert \theta\vert<\pi/2$, $\Stheta \fgo (1)$ admits the (absolutely convergent) integral representation
\begin{equation}
 \Stheta \fgo (1)
 =\frac1{\Gamma(\beta)}
 \int_0^{e^{-i\theta}\infty}
 e^{-u}u^{\beta-1}(1-u)^{-\alpha}\,du.
 \label{eq:plusbeta}
\end{equation}
The value of the integral depends on $\theta\neq 0$ only by the sign of $\theta$. Moreover,
\begin{equation}
 \Stheta \fgo (1) =\frac1{\Gamma(\beta)}\Big(
 \int_0^1 e^{-u}u^{\beta-1}(1-u)^{-\alpha}\,du
 +e^{-\eps i\pi\alpha}\int_1^{+\infty} e^{-u}u^{\beta-1}(u-1)^{-\alpha}\,du
 \Big) \label{eq:lateralbeta}
\end{equation}
where $\eps\in\{-1,1\}$ is the sign of $\theta$.

\medskip
 $(ii)$ For any $\theta$ such that $\vert \theta\vert<\pi/2$, $\Stheta\tilde \fgo (1)$ admits the (absolutely convergent) integral representation
\begin{align}
 \Stheta\tilde \fgo (1)
 &=-\frac1{\Gamma(\beta)}
 \int_0^{+\infty} e^{-u}u^{\beta-1}(1+u)^{-\alpha}\,du.
 \label{eq:minusbeta}
\end{align}
It is independent of $\theta$. 
\end{lem}

\begin{proof} Throughout the proof, the assumptions $0<\Re(\alpha), \Re(\beta)<1$ ensure that all integrals are absolutely convergent and that we can perform for instance change of variables and inversion of integral signs. 

$(i)$ We use Euler's classical representation 
\begin{equation}\label{eq:eulerbeta}
 {}_2F_1(\alpha,\beta;1;\xi)
 =\frac1{\Gamma(\beta)\Gamma(1-\beta)}
 \int_0^1 \frac{t^{\beta-1}(1-t)^{-\beta}}
{(1-\xi t)^{\alpha}}\,dt,
\end{equation}
initially valid for $\vert \xi\vert <1$, but the integral defines the analytic continuation to $\xi\in \mathbb C\setminus [1, +\infty)$. Writing $\xi=e^{-i\theta}s$ in \eqref{eq:BLplusbis} gives
\begin{equation}\label{eqpareil}
 \Stheta \fgo (1)
 =\frac{e^{-i\theta}}{\Gamma(\beta)\Gamma(1-\beta)}
 \int_0^{+\infty}\!\int_0^1
 e^{-e^{-i\theta}s}t^{\beta-1}(1-t)^{-\beta}
 (1-e^{-i\theta}st)^{-\alpha}\,dt\,ds.
\end{equation}
The classical change of variables
$x=st$ and $y=s(1-t)$ maps $(0,+\infty)\times(0,1)$ bijectively onto $(0,+\infty)^2$, with
$$
s=x+y,\qquad t=\frac{x}{x+y},\qquad
 ds\,dt=\frac{dx\,dy}{x+y}.
$$
Moreover
$
 t^{\beta-1}(1-t)^{-\beta}
 =x^{\beta-1}y^{-\beta}(x+y).
$
Hence, 
\begin{equation*}
 \Stheta \fgo (1)
 =\frac{e^{-i\theta}}{\Gamma(\beta)\Gamma(1-\beta)}
 \left(\int_0^{+\infty} e^{-e^{-i\theta}x}
 x^{\beta-1}(1-e^{-i\theta}x)^{-\alpha}\,dx\right)\times\left(\int_0^{+\infty} e^{-e^{-i\theta}y}y^{-\beta}\,dy\right).
\end{equation*}
Since $\vert \theta\vert<\pi/2$, we have 
$$
 \int_0^{+\infty} e^{-e^{-i\theta}y}y^{-\beta}\,dy
 =\Gamma(1-\beta)e^{i\theta(1-\beta)}.
$$
Hence
$$
 \Stheta \fgo (1)
 =\frac{e^{-i\theta\beta}}{\Gamma(\beta)}
 \int_0^{+\infty} e^{-e^{-i\theta}x}
 x^{\beta-1}(1-e^{-i\theta}x)^{-\alpha}\,dx.
$$
With $u=e^{-i\theta}x$ this becomes \eqref{eq:plusbeta}. By Cauchy's theorem, \eqref{eq:plusbeta} depends on $\theta$ only by its sign because the singularities of the integrand are in $[0,+\infty)$.

We next let $\theta$ tend to $0^+$ or $0^-$. Near $u=1$ the integrand has size $|1-u|^{-\Re(\alpha)}$; hence both integrals converge under the assumptions on $\alpha$ and $\beta$. For $\Re(u)>1$, when the contour approaches from the upper half-plane, $1-u$ approaches the negative real axis from below, so that
$(1-u)^{-\alpha}=e^{i\pi\alpha}(u-1)^{-\alpha}$ for $u>1$; 
when the contour approaches from below, the factor is $e^{-i\pi\alpha}$. This proves \eqref{eq:lateralbeta}. 

\medskip

The proof of $(ii)$ is similar. We begin with Eq. \eqref{eq:BLpluster} and apply Euler's identity \eqref{eq:eulerbeta}. Letting $\xi=e^{-i\theta}s$ we obtain
$$
 \Stheta\tilde \fgo (1)
 =-\frac{e^{-i\theta}}{\Gamma(\beta)\Gamma(1-\beta)}
 \int_0^{+\infty}\!\int_0^1
 e^{-e^{-i\theta}s}t^{\beta-1}(1-t)^{-\beta}
 (1+e^{-i\theta}st)^{-\alpha}\,dt\,ds.
$$
Up to the sign coming from Eq. \eqref{eq:BLpluster}, this is exactly the same right hand side as in Eq. \eqref{eqpareil}, except that we have $ (1+e^{-i\theta}st)^{-\alpha}$ instead of $ (1-e^{-i\theta}st)^{-\alpha}$. The rest of the proof is completely similar, except that the integrand has no pole in the positive real axis: the sign of $\theta$ does not matter, and $\theta$ can be chosen equal to zero.
\end{proof}

\subsection{Application: the value $ {}_1F_1[a;b;1]$} \label{secexemple}

Let $f(x) := {}_1F_1(a;b;x)$ with $a,b\in\Q\setminus\Z$ such that $a-b\not\in\Z$. Let $\theta$ be a real number with $0<|\theta|<\pi$; denote by $\eps\in\{-1,1\}$ the sign of $\theta$. Then we have \cite[Eq. (13.7.2)]{DLMF}
$$
f(x)\approx \frac{\Gamma(b)}{\Gamma(b-a)}e^{\eps i\pi a} x^{-a} {}_2F_0[a,1+a-b; ;-1/x] + \frac{\Gamma(b)}{\Gamma(a)} e^{x} x^{a-b} {}_2F_0[1-a,b-a; ;1/x]
$$
in a large sector bisected by $\theta$. 

\bigskip

We denote by $\Stheta$ the Borel-Laplace summation (i.e., $1$-summation) in a large sector bisected by $\theta$; in other words, $\Stheta$ is the inverse of the asymptotic expansion map. Then with the notation of Theorem~\ref{thdcp} we have 
$$f(x) = f_0(x)+e^x f_1(x)
$$
with
$$f_0(x) = \frac{\Gamma(b)}{\Gamma(b-a)}e^{\eps i\pi a} x^{-a} \Stheta\, {}_2F_0[a,1+a-b; ;-1/x]
$$
and
$$f_1(x) = \frac{\Gamma(b)}{\Gamma(a)} x^{a-b} \Stheta\, {}_2F_0[1-a,b-a; ;1/x].
$$
Since anti-Stokes directions are exactly the multiples of $\pi$, $\Stheta$ depends only on the sign $\eps$ of $\theta$ (recall that we assume $0<|\theta|<\pi$); so do $f_0$ and $f_1$. From now on, we assume that $0<|\theta|<\pi/2$. Then we have
$$
\Phi_\theta(f(1)) = f_0(1)+ f_1(1)e^x.
$$

Let us give explicit integral expressions of $f_0(1)$ and $f_1(1)$ using Lemma~\ref{lem:1summation}. With this aim in mind, we assume that $0<a<1$ and $0<b-a<1$ and consider the following functions: 
$$
\tilde \fgo _0(1/x) := \frac{-1}x {}_2F_0[a,1+a-b; ;-1/x] \quad \mbox{ and }\quad 
\fgo _1(1/x) := \frac1x {}_2F_0[1-a,b-a; ;1/x],
$$
so that 
$$f_0(x) = - \frac{\Gamma(b)}{\Gamma(b-a)}e^{\eps i\pi a} x^{1-a} \Stheta\, \tilde \fgo _0(1/x)\quad \mbox{ and }\quad 
f_1(x) = \frac{\Gamma(b)}{\Gamma(a)} x^{1+a-b} \Stheta\, \fgo _1(1/x) .
$$
Lemma~\ref{lem:1summation} yields 
$$ \Stheta\,\tilde \fgo _0(1) = \frac{-1}{\Gamma(1+a-b)} \int_0^{+\infty} e^{-u}u^{a-b}(1+u)^{-a}\,du$$
and
$$ \Stheta\, \fgo _1(1)=\frac{1}{\Gamma(b-a)}\Big(
 \int_0^1 e^{-u}u^{b-a-1}(1-u)^{a-1}\,du
 +e^{-\eps i\pi(1-a)}\int_1^{+\infty} e^{-u}u^{b-a-1}(u-1)^{a-1}\,du
 \Big)
$$
so that 
\begin{equation}\label{eq:intf0}
f_0(1) = \frac{\Gamma(b)}{\Gamma(b-a)\Gamma(1+a-b)}e^{\eps i\pi a} 
 \int_0^{+\infty} e^{-u}u^{a-b}(1+u)^{-a}\,du
\end{equation}
and
\begin{equation}\label{eq:intf1}
f_1(1) = \frac{\Gamma(b)}{\Gamma(a)\Gamma(b-a)}\Big(
 \int_0^1 e^{-u}u^{b-a-1}(1-u)^{a-1}\,du
 +e^{-\eps i\pi(1-a)}\int_1^{+\infty} e^{-u}u^{b-a-1}(u-1)^{a-1}\,du
 \Big).
\end{equation}
In particular, both $f_0(1)$ and $ f_1(1)$ are non-zero: using Corollary~\ref{corpasunite} we deduce that $f(1)$ is not a unit of $\EE$. 

\bigskip

We conclude this section with a few numerical examples (performed with Maple 2023) to  illustrate the equation $f(1) = f_0(1)+e f_1(1)$. The value $f(1)$ is computed by an evaluation of hypergeom$([a],[b],1)$, the values $f_0(1)$ and $f_1(1)$ are computed by means of the integral expressions \eqref{eq:intf0} and \eqref{eq:intf1}, which Maple computes with 10 digits after the decimal point in less than a second. 

With $a=1/3$ and $b=1/2$, we have 
$$
f_0(1)\simeq 0.1353357497+0.0234408394i, \; 
f_1(1) \simeq 0.7161782862-0.8623402924i, $$  
$$f_0(1)+e f_1(1) 
\simeq 2.0821101711 \simeq f(1).
$$
With $a=1/4$ and $b=1/3$, we have 
$$
f_0(1)\simeq 0.1441081810+0.1441081810i, \; 
f_1(1) \simeq 0.7741774418-0.0530144371i, 
$$  
$$
f_0(1)+e f_1(1) 
\simeq 2.2485406533 \simeq f(1).
$$

\subsection{The Bessel function} \label{subsecbessel}

Consider the Bessel function $J_0(x):=\sum_{n=0}^\infty (-1)^n (x/2)^{2n}/n!^2$, and let $\theta\in\R$ with $-\pi/2<\theta<\pi/2$. Then the asymptotic expansion of $J_0$ in a large sector bisected by $\theta$ (see \cite[Eq. (10.17.3)]{DLMF}) can be written as $f_i(x)e^{ix}+f_{-i}(x)e^{-ix}$ with 
$$
f_i(x) \approx \widehat f_i(x) := \frac1{\sqrt{2i\pi}}x^{-1/2} \sum_{n=0}^\infty \frac{(1/2)_n^2}{(2i)^n n!}x^{-n}
$$
and the same formula for $f_{-i}(x)$ (up to replacing everywhere $i$ by $-i$). 
Now $\widehat f_i(x) $ and $\widehat f_{-i}(x) $ are coefficientwise complex conjugates. Choosing $\theta=0$, we deduce that their Borel-Laplace sums in a large sector bisected the real axis take conjugate values at real values of $x$. Accordingly, $f_i(1)$ and $f_{-i}(1)$ are complex conjugates. Therefore they are non-zero, since otherwise $J_0(1)$ would be zero (and it is transcendental due to Siegel's theorem \cite{siegel}):
 Corollary~\ref{corpasunite} yields $J_0(1)\not\in \EE^\times$.

\section{Proof of Theorem~\ref{th1}}\label{sec5}

\begin{proof}[Proof of Theorem~\ref{th1}] It follows from Theorem~\ref{th2} that $e+1$ is not a unit of $\EE$, so that $\EE$ is not a field. 
We proved in \S \S \ref{secexemple} and~\ref{subsecbessel} that $ {}_1F_1(a;b;1)$ (with rationals $a$, $b$ such that $0<a<1$ and $0<b-a<1$) and $J_0(1)$ are not units either.

\bigskip

To prove that $\EE$ is not a factorial ring, we notice that for any $n\geq 1$,
\begin{equation}\label{eqnonfacto}
e-1 = \prod_{k=1}^n \Big( e^{1/n}-e^{2ik\pi/n}\Big) 
\end{equation}
with $ e^{1/n}-e^{2ik\pi/n} \in\EE\setminus\EE^\times$ for any $k$ using Theorem~\ref{th2}. If $\EE$ is factorial, for any $x\in\EE\setminus\{0\}$ we may denote by $\ell(x)$ the sum of its $p$-adic valuations (taken over all irreducible elements $p$, up to association). Then $\ell(x)\geq 1$ for any non-unit $x$, and $\ell(xy)=\ell(x)+\ell(y)$ for any non-zero $x$, $y$. Accordingly, when writing a non-zero element $x$ as a product of non-units, the number of factors is at most $\ell(x)$. This contradicts Eq. \eqref{eqnonfacto} for $n>\ell(e-1)$: $\EE$ is not a factorial ring.

\bigskip

To prove that $\EE$ is not noetherian, we denote by $I_n$ the principal ideal of $\EE$ generated by $\exp(1/2^{n})-1$. Then $(I_n)_{n\geq 0}$ is an infinite increasing sequence of principal ideals. Indeed, for any $n$ we have $$(\exp(1/2^{n+1})-1)(\exp(1/2^{n+1})+1)=\exp(1/2^n)-1$$ 
with $\exp(1/2^{n+1})+1\in\EE\setminus\EE^\times$ using Theorem~\ref{th2}, so that $I_n\subsetneq I_{n+1}$.
\end{proof}

\section{Open questions}\label{secopen}

An arithmetic Nilsson-Gevrey series of order 1 is a finite sum
$$
\sum_{\alpha,k,\ell } \lambda_{\alpha,k,\ell }\, x^\alpha \, (\log x)^k \, y_{\alpha,k,\ell}(x)
$$
where the indices $\alpha$ are rationals, $k$ and $\ell$ are non-negative integers, the $y_{\alpha,k,\ell}$ are 
 \antie-functions and the $\lambda_{\alpha,k,\ell }$ are complex numbers. The set of arithmetic Nilsson-Gevrey series of order 1 is denoted by $\NGA_1\{x\}$. 

\bigskip

We also recall the following notation:
\begin{itemize}
\item $\GG$ is the ring of values at algebraic points of (analytic continuations) of $G$-functions (see \cite[\S 2.2]{gvalues}). This ring is very important because, conjecturally, $\GG=\calP[1/\pi]$ where $\calP$ is the ring of periods (in the sense of Kontsevich-Zagier \cite{KZ}, say).
\item ${\bf S}$ is the $\GG$-module generated by all the values of derivatives of the Gamma function at rational points (see \cite[\S 2]{ateo}). It is also the $\GG[\gamma]$-module generated by all the values of $\Gamma$ at rational points, where $\gamma$ is Euler's constant; it is a ring.
\item $ \DD$ is the set of all values 
$\fgo_\theta(1)$ where $\fgo$ is an \antie-function (here $\theta=0$ if it is not an anti-Stokes direction of $\fgo$, and $\theta>0$ is very small otherwise.) Then $ \DD$ is a countable sub-ring of $\mathbb C$ that contains $\Qbar$ (see \cite{ssmixte}, and \cite[\S 4.3]{fuchsien} where it is denoted by {\bf \tantie}).
\item $\VV$ is the ${\bf S}$-module generated by $\E$; it is a ring (see \cite{fuchsien}). It consists in all finite sums of products $G(1)f(1)\Gamma(r)\gamma^j$ where $G$ is the analytic continuation of a $G$-function, $f$ is an $E$-function, $r\in\Q\setminus\Z_{\leq 0}$, $j\in\N$ (see \cite{gvalues} and \cite{ateo}). It contains all values at algebraic points of Airy's oscillating integral, and that of Bessel's functions $J_\alpha(z)$ with $\alpha\in\Q$. Conjecturally, $\VV$ is the ring generated by exponential periods and $1/\pi$.
\end{itemize}

We have proved \cite[Theorem 3]{fuchsien} that $\VV=\sum_{\rho\in\Qbar} e^\rho \VV_0$ where $\VV_0$ is the ${\bf S}$-module generated by $\DD$. Since $\EE\subset\VV$, any element $\xi\in\EE$ has a decomposition in the sum $\sum_{\rho\in\Qbar} e^\rho \VV_0$. The construction of \S \ref{secmorph} with $\theta=0$ gives such a decomposition: letting $\Phi_0(\xi)=\sum_{\rho\in\Qbar}\lambda_\rho e^{\rho x}$ we have $\xi = \sum_{\rho\in\Qbar}\lambda_\rho e^{\rho}$, with $\lambda_\rho\in\VV_0$ for any $\rho$ using \cite[Theorem 2]{ateo}.

\bigskip

An interesting question would be to describe explicitly the image of $\Phi_0$: it 
 is contained in the set of $\sum_{\rho\in\Qbar}\lambda_\rho e^{\rho x}$ where the $\lambda_\rho$ belong to $\VV_0$, but is there equality?

\bigskip

Another question would be to study factorization in the ring $\EE$. To begin with, let us recall how factorization works in the ring $\mathcal{R}$ of exponential polynomials with constant complex coefficients and algebraic phases (the situation would be the same with algebraic coefficients and/or complex phases). Units of $\mathcal{R}$ are elements of the form $\lambda e^{\alpha x}$ with $\lambda\in\C^\times$ and $\alpha\in\Qbar$. Irreducible elements are those non-units that cannot be written as a product of two non-units. Following Ritt, simple elements are the non-zero and non-units of the form $ e^{\alpha x} P( e^{\beta x})$
with $\alpha,\beta\in\Qbar$ and $P\in\C[X]$. The {\em support} of such a simple exponential polynomial is the $1$-dimensional $\Q$-vector space spanned by~$\beta$.

With these notations, any non-zero $f\in\calR$ can be written in a unique way (up to units) as a product of irreducible elements, and simple ones with pairwise distinct supports. This result is due to Ritt~\cite{ritt1927}; it was generalized by MacColl~\cite{maccoll} to the ring of exponential polynomials with polynomial coefficients. A conjectural analogue for the ring $\calE$ of $E$-functions appears in \cite{ezeros}.

We expect factorization properties of the ring $\EE$ to be similar. We suggest to call 
{\em simple} a non-zero and non-unit element of the form 
$$ e^\alpha P( e^{\beta }) \mbox{ with $\alpha,\beta \in\Qbar$ and $P\in\Qbar[X]$},$$
and to define its {\em support} as the $\Q$-vector space spanned by $\beta$. Is it true that any non-zero element of $\EE$ can be written in a unique way (up to units) as a product of irreducible elements, and simple ones with pairwise distinct supports?

\noindent St\'ephane Fischler, Universit\'e Paris-Saclay, CNRS, Laboratoire de math\'ematiques d'Orsay, 91405 Orsay, France.

\medskip

\noindent Tanguy Rivoal, Universit\'e Grenoble Alpes, CNRS, Institut Fourier, CS 40700, 38058 Grenoble cedex 9, France.

\bigskip

\noindent Keywords: $E$-functions, Special value,
Asymptotic expansion, Borel-Laplace summation

\bigskip

\noindent MSC 2020: 11J91, 33E30, 34M40, 44A10


\begin{thebibliography}{10}

\bibitem{andre} Y. Andr\'e, S\'eries Gevrey de type arithm\'etique I. Th\'eor\`emes de puret\'e et de dualit\'e, {\em Ann. of Math.} {\bf 151}.2 (2000), 705--740.

\bibitem{andreasens} Y. Andr\'e, Solution algebras of differential equations and quasi-homogeneous varieties: a new differential Galois correspondence, {\em Ann. Sci. \'Ec. Norm. Sup\'er.} {\bf 47}.2 (2014), 449–-467.

\bibitem{beukers} F.~Beukers, A refined version of the Siegel-Shidlovskii theorem, {\em Ann. of Math.} {\bf 163}.1 (2006), 369--379.

\bibitem{DLMF} NIST Digital Library of Mathematical Functions. https://dlmf.nist.gov/, Release 1.2.8 of 2026-09-15. F. W. J. Olver, A. B. Olde Daalhuis, D. W. Lozier, B. I. Schneider, R. F. Boisvert, C. W. Clark, B. R. Miller, B. V. Saunders, H. S. Cohl, and M. A. McClain, eds.

\bibitem{gvalues} S. Fischler, T. Rivoal, On the values of $G$-functions, {\em Commentarii Math. Helv.} {\bf 89}.2 (2014), 313--341.

\bibitem{ateo} S. Fischler, T. Rivoal, Arithmetic theory of $E$-operators, {\em J. \'Ec. Polytech. -- Math\'e\-ma\-ti\-ques} {\bf 3} (2016), 31--65. 

\bibitem{fuchsien} S. Fischler, T. Rivoal, Microsolutions of differential operators and values of arithmetic Gevrey series, {\em American J. of Math.}
{\bf 140}.2 (2018), 317--348.

\bibitem{eliouville} S. Fischler, T. Rivoal, Values of $E$-functions are not Liouville numbers, {\em J. \'Ec. Polytech. -- Ma\-th\'e\-matiques} {\bf 11} (2024), 1--18.

\bibitem{ssmixte} S. Fischler, T. Rivoal, Relations between values of arithmetic Gevrey series, 
and applications to values of the Gamma function, {\em Journal of Number Theory} {\bf 261} (2024), 36--54.

\bibitem{ezeros} S. Fischler, T. Rivoal, 
Zeros of $E$-functions and of exponential polynomials defined over $\Qbar$, {\em Rend. Circ. Mat. Palermo, II. Ser} {\bf 75} (2026), article \# 158, 21 pages.

\bibitem{muegaledeux} S. Fischler, T. Rivoal, Rational approximations to values of $E$-functions, preprint arXiv 2312.12043, 32 pages, submitted.

\bibitem{Fresannotes} J. Fresán, Lectures on arithmetic Gevrey series, course notes, 
9th November 2023, available at {\tt http://javier.fresan.perso.math.cnrs.fr/nachdiplom.html}

\bibitem{FJ} J. Fresán, P.~Jossen, Exponential motives, Preliminary version available at {\tt http://javier.fresan.perso.math.cnrs.fr/expmot.pdf}

\bibitem{KZ} M. Kontsevich, D. Zagier, Periods, in: 
Mathematics Unlimited -- 2001 and beyond, Springer, 2001, 771--808. 

\bibitem{lepetit} G. Lepetit, Le th\'eor\`eme d'Andr\'e-Chudnovsky-Katz au sens large, {\em North-West. Eur. J. Math.} {\bf 7} (2021), 83--149.

\bibitem{maccoll} L. A. MacColl, A factorization theory for polynomials in $x$ and in functions $e^{\alpha x}$, {\em Bull. Amer. Math. Soc.} {\bf 41}.2 (1935), 104--109. 

\bibitem{Ramis} J. P. Ramis, {\em S\'eries Divergentes et Th\'eories Asymptotiques}, Panoramas et Synth\`eses {\bf 21}
(1993), Soc. Math. France, Paris.

\bibitem{ritt1927} J. F. Ritt, A factorization theory for functions $\sum_{i=1}^n a_i e^{\alpha_i x}$, {\em Trans. Amer. Math. Soc.} {\bf 29} (1927), 584--596. 


\bibitem{XUPS} T. Rivoal, Les $E$-fonctions et $G$-fonctions de Siegel, in: Périodes et transcendance, Journées X-UPS 2019, Éditions de l'École polytechnique, 2024, 189--284.

\bibitem{RiRopreprint} T. Rivoal, J. Roques, $E$-functions of order 2 and units of $E$-values, preprint available at {\tt http://rivoal.perso.math.cnrs.fr/articles/unitse.pdf}, 2015, 16 pages.

\bibitem{siegel} C. Siegel, \"Uber einige Anwendungen diophantischer Approximationen, vol. 1 S. {\em Abhandlungen Akad.}, Berlin, 1929.


\end{thebibliography}
\end{document}